\documentclass[11pt,reqno]{amsart}
{\theoremstyle{plain}
 \newtheorem{theorem}{Theorem}
 \newtheorem{corollary}{Corollary}
 
\newtheorem{lemma}{Lemma}
}

\DeclareMathOperator{\re}{Re}

\begin{document}

\title[On the constant factor in several asymptotic estimates]{On the constant factor in several related asymptotic estimates}

\begin{abstract}
We establish formulas for the constant factor in several asymptotic estimates related to the distribution of integer and polynomial divisors. 
The formulas are then used to approximate these factors numerically.
\end{abstract}

\author{Andreas Weingartner}
\address{ 
Department of Mathematics,
351 West University Boulevard,
 Southern Utah University,
Cedar City, Utah 84720, USA}
\email{weingartner@suu.edu}
%\date{May 17, 2017}
%\subjclass[2010]{11N25, 11N37}
\maketitle

\section{Introduction}
A number of asymptotic estimates \cite{ PTW, PDD, DPD}, related to the distribution of divisors of integers and of polynomials,
contain a constant factor that is as yet undetermined. 
In this note, we give an explicit formula for this constant as the sum of an infinite series.
As a result, we are able to approximate this factor numerically in several instances and improve some of the error terms in  \cite{ PDD, DPD}.
For more extensive background information, we refer the reader to \cite{ PTW, PDD, DPD} and the references therein.

We begin by recalling the general setup from \cite{PDD}.
Let $\theta$ be a real-valued arithmetic function. 
Let $\mathcal{B}=\mathcal{B}_\theta$ be the set of positive integers containing $n=1$ and all those $n \ge 2$ with prime factorization $n=p_1^{\alpha_1} \cdots p_k^{\alpha_k}$, $p_1<p_2<\ldots < p_k$, which satisfy 
\begin{equation*}\label{Bdef}
p_{i} \le \theta\big( p_1^{\alpha_1} \cdots p_{i-1}^{\alpha_{i-1}} \big) \qquad (1\le i \le k).
\end{equation*}
We write $B(x)$ to denote the number of integers $n\le x$ in $\mathcal{B}$.
Theorem 1.2 of \cite{PDD} states that, if 
\begin{equation}\label{theta}
\theta(1)\ge 2, \quad   n\le  \theta(n) \le A n (\log 2n)^a (\log\log 3n)^b \quad (n\ge 2)
 \end{equation}
for suitable constants $A\ge 1$, $a<1$, $b$, then
\begin{equation}\label{Basymp}
B(x)=\frac{c_\theta x}{\log x} \Big\{1+O_\theta\big( (\log x)^{a-1} (\log\log x)^b  \big)\Big\},
\end{equation}
for some positive constant $c_\theta$. This result still holds if $a=1$ and $b<-1$, provided $b$ is replaced by $b+1$ in
the error term of \eqref{Basymp}.

\begin{theorem}\label{thm1}
Assume $\theta$ satisfies \eqref{theta}.
The constant $c_\theta$ in \eqref{Basymp} is given by
$$
c_\theta = \frac{1}{1-e^{-\gamma}} \sum_{n\in \mathcal{B}} \frac{1}{n}  \Biggl( \sum_{p\le \theta(n)}\frac{\log p}{p-1} - \log n\Biggr) \prod_{p\le \theta(n)} \left(1-\frac{1}{p}\right),
$$
where $\gamma$ is Euler's constant and $p$ runs over primes.
\end{theorem}

\subsection{Practical numbers}
A well known example is the set $\mathcal{P}$ of \emph{practical numbers} \cite{Sri}, i.e. integers $n$ with the property that every natural number $ m \le n$ can be expressed as a sum of distinct positive divisors of $n$. Stewart \cite{Stew} and Sierpinski \cite{Sier} found that $\mathcal{P}=\mathcal{B}_\theta$ if $\theta(n)=\sigma(n)+1$, where $\sigma(n)$ denotes the sum of the positive divisors on $n$. Since $n+1\le \sigma(n)+1 \ll n \log\log 3n$, 
\eqref{Basymp} shows that the number of practical numbers up to $x$ satisfies
\begin{equation}\label{Pasymp}
P(x)=\frac{c x}{\log x} \left\{1+O\left( \frac{\log \log x}{\log x}  \right)\right\},
\end{equation}
for some $c>0$. Theorem \ref{thm1} states that
\begin{equation}\label{pracon}
c= \frac{1}{1-e^{-\gamma}} \sum_{n\in \mathcal{P}} \frac{1}{n}  \Biggl( \sum_{p\le \sigma(n)+1}\frac{\log p}{p-1} - \log n\Biggr) \prod_{p\le \sigma(n)+1} \left(1-\frac{1}{p}\right),
\end{equation}
from which we will derive the following bounds.

\begin{corollary}\label{Pcor}
The constant $c$ in \eqref{Pasymp} satisfies $1.311 < c <1.693$.
\end{corollary}
Corollary \ref{Pcor} is consistent with the empirical estimate $c \approx 1.341$ given by Margenstern \cite{Mar}.
The lack of precision in Corollary \ref{Pcor}, when compared with Corollary \ref{Dcor}, is due to the fact that $\theta(n)/n$ is not bounded when $\theta(n)=\sigma(n)+1$, which makes
it more difficult to estimate the tail of the series \eqref{pracon}.

\subsection{The distribution of divisors}
Another example is the set $\mathcal{D}_t$ of integers with \emph{$t$-dense} divisors \cite{Saias1, Ten86}, i.e. integers $n$ whose divisors $1=d_1 < d_2 < \ldots < d_{\tau(n)} = n$
satisfy $d_{i+1}\le t d_i$ for all $1\le i < \tau(n)$. 
Tenenbaum \cite[Lemma 2.2]{Ten86} showed that these integers are exactly the members of $\mathcal{B}_\theta$ if $\theta(n)=tn$.
When $t\ge 2$ is fixed, \eqref{Basymp} implies that the number of such integers up to $x$ satisfies
\begin{equation}\label{Dasymp}
D(x,t)=\frac{c_t \, x}{\log x} \left\{1+O_t\left( \frac{1}{\log x}  \right)\right\}.
\end{equation}
Theorem \ref{thm1} yields
\begin{equation}\label{tdensecon}
c_t= \frac{1}{1-e^{-\gamma}} \sum_{n\in \mathcal{D}_t} \frac{1}{n}  \Biggl( \sum_{p\le t n }\frac{\log p}{p-1} - \log n\Biggr) \prod_{p\le t n} \left(1-\frac{1}{p}\right).
\end{equation}
Comparing \eqref{Dasymp} with \cite[Cor. 1.1]{PDD}, we find that the constant factor $\eta(t)$ appearing in \cite[Thm. 1.3]{PDD} is given by $\eta(t)=c_t(1-e^{-\gamma})/\log t$.

We can now give numerical approximations for $c_t$ (and hence $\eta(t)$) based on \eqref{tdensecon}. 
The details behind these calculations will be described in Section \ref{CorThm1Sec}.
\begin{corollary}\label{Dcor}
Table \ref{table1} shows values of the factor $c_t$ appearing in \eqref{Dasymp}.
\end{corollary}

\begin{table}[h]
  \begin{tabular}{ | c | c | }
    \hline
    $\ t \ $ & $c_t$ \\ \hline
    $2$ &  1.2248...  \\ \hline
    $e$ & 1.5242...  \\ \hline
    $3$ &  2.0554... \\ \hline
    $4$ & 2.4496... \\ \hline
    $5$ & 2.9541...  \\ \hline
     \end{tabular}
     \quad
     \begin{tabular}{ | c | c | }
    \hline
    $\ t \ $ & $c_t$ \\ \hline   
     $6$ & 3.247... \\ \hline
    $7$ & 3.644... \\ \hline
    $8$ & 3.850... \\ \hline
     $9$ & 4.041... \\ \hline
     $10$ & 4.227... \\ \hline
      \end{tabular}
       \quad
     \begin{tabular}{ | c | c | }
    \hline
    $\ t \ $ & $c_t$ \\ \hline   
    $20$ & 5.742... \\ \hline
     $40$ & 7.210... \\ \hline
     $60$ & 8.113... \\ \hline
     $80$ & 8.761... \\ \hline
       $100$ & 9.248... \\ \hline
      \end{tabular}
       \quad
     \begin{tabular}{ | c | c | }
    \hline
    $\ t \ $ & $c_t$ \\ \hline   
      $10^3$ & 14.449... \\ \hline
      $10^4$ & 19.689... \\ \hline
      $10^5$ & 24.937... \\ \hline
      $10^6$ & 30.187... \\ \hline
      $10^7$ & 35.43.... \\ \hline
      \end{tabular}
          \medskip
     \caption{Truncated values of $c_t$ derived from Theorem \ref{thm1}.}\label{table1}
\end{table}

For example, the number of integers $n\le x$, which have a divisor in the interval $(y,2y]$ for every $y\in [1,n)$, is
$$D(x,2)=1.2248... \frac{ x}{\log x} \left\{1+O\left( \frac{1}{\log x}  \right)\right\},$$
so that these integers are about $22.5\%$ more numerous than the primes.

Corollary 1.1 of \cite{PDD} gives an estimate for $D(x,t)$ which holds uniformly in $t$. 
It states that, uniformly for $x\ge t\ge 2$, 
\begin{equation}\label{Dasymp2}
D(x,t)=\frac{c_t \, x}{\log (tx)} \left\{1+O\left(\frac{1}{\log x}+\frac{\log^2 t}{\log^2 x}\right)\right\},
\end{equation}
where $c_t=(1-e^{-\gamma})^{-1}  \log t + O(1)$. We can improve the estimate for $c_t$ with the help of Theorem \ref{thm1}.
\begin{corollary}\label{Dcor2}
Let $c_t$  be the factor in \eqref{Dasymp} and \eqref{Dasymp2}. Define $\delta_t$ implicitly by
\begin{equation}\label{cteq}
c_t = \frac{\log(t e^{-\gamma}) +\delta_t}{1-e^{-\gamma}}.
\end{equation}
We have 
$ \delta_t  \ll \exp(-\sqrt{\log t})$ and 
\begin{equation}\label{deltat}
| \delta_t |  \le  \frac{0.084}{\log^2 t}\qquad (t \ge 2^{25}).
\end{equation}
Assuming the Riemann hypothesis, we have 
$|\delta_t| \le \frac{\log^2 t}{7\sqrt{t}}$ for $t\ge 55$.
\end{corollary}
\begin{table}[h]
     \begin{tabular}{ | c | l | }
    \hline
    $\ t \ $ & $c_t$ \\ \hline   
    $10^8$ &  40.68...  \\ \hline
    $10^9$ & 45.93...  \\ \hline
    $10^{10}$ & 51.189...  \\ \hline
     \end{tabular}
               \quad
     \begin{tabular}{ | c | l | }
    \hline
    $\ t \ $ & $c_t$ \\ \hline   
    $10^{20}$ & 103.69...  \\ \hline
    $10^{30}$ & 156.200...  \\ \hline
    $10^{40}$ & 208.7063...  \\ \hline
             \end{tabular}
               \quad
     \begin{tabular}{ | c | l | }
    \hline
    $\ t \ $ & $c_t$ \\ \hline   
    $10^{60}$ & 313.7176...  \\ \hline
    $10^{80}$ & 418.7289...  \\ \hline
    $10^{100}$ & 523.7401...  \\ \hline
             \end{tabular}
   \medskip
     \caption{Truncated values of $c_t$ derived from \eqref{cteq} and \eqref{deltat}.}\label{table2}
\end{table}
Assuming the Riemann hypothesis, the last entry in Table \ref{table2} is
$$c_{10^{100}} = 523.74019053615422813260729554671054989578274943...$$

Combining the estimate \eqref{Dasymp2} with Corollary \ref{Dcor2}, we obtain the following improvement of \cite[Corollary 1.2]{PDD}.
\begin{corollary}\label{Dcor3}
Uniformly for $x\ge t\ge 2$, we have
\begin{equation*}\label{Dasymp3}
D(x,t)=\frac{x \log(t e^{-\gamma} ) }{(1-e^{-\gamma})\log (tx)} \left\{1+O\left(\frac{1}{\log x}+\frac{\log^2 t}{\log^2 x}+\frac{1}{\exp(\sqrt{\log t})}\right)\right\}.
\end{equation*}
The error term $O(\exp(-\sqrt{\log t}))$ can be replaced by $O(\log(t) /\sqrt{t})$ if the Riemann hypothesis holds.
\end{corollary}

\subsection{$\varphi$-practical numbers}\label{phiprasec}

An integer $n$ is called \emph{$\varphi$-practical} \cite{Thom} if $X^n-1$ has divisors in $\mathbb{Z}[X]$ of every degree up to $n$. 
The name comes from the fact that $X^n-1$ has this property if and only if each natural number $m\le n$ is a subsum 
of the multiset $\{\varphi(d): d|n\}$, where $\varphi$ is Euler's function. These numbers were first
studied by Thompson \cite{Thom}, who showed that their counting function $P_\varphi(x)$ has order of magnitude $x/\log x$.
Pomerance, Thompson and the author \cite{PTW} established the asymptotic result
\begin{equation}\label{Phiasymp}
P_\varphi(x)=\frac{C x}{\log x} \left\{1+O\left( \frac{1}{\log x}  \right)\right\},
\end{equation}
for some positive constant $C$.

Although the set of $\varphi$-practical numbers, say $\mathcal{A}$, is not exactly an example of 
a set $\mathcal{B}_\theta$ as described earlier, Thompson \cite{Thom} showed that 
$ \mathcal{B}_{\theta_1} \subset \mathcal{A} \subset \mathcal{B}_{\theta_2},$
where $\theta_1(n)=n+1$ and $\theta_2(n)=n+2$. $ \mathcal{B}_{\theta_1}$ is the set of even 
$\varphi$-practical numbers, while the integers in $\mathcal{B}_{\theta_2}$ are called \emph{weakly $\varphi$-practical} in \cite{Thom}.
We can use Theorem \ref{thm1} to estimate the constants $c_{\theta_1}$ and $c_{\theta_2}$ .
\begin{corollary}\label{Phicor}
If $\theta(n)=n+1$,  $c_\theta = 0.8622...$. If $\theta(n)=n+2$, $c_\theta = 1.079...$.
\end{corollary}
It follows that the constant $C$ in \eqref{Phiasymp} satisfies $0.8622<C<1.080$. 
Our goal is to give a formula for the exact value of $C$.
As the proof of \eqref{Phiasymp} in \cite{PTW} is more general and applies to other similar sequences,
so does Theorem \ref{thm2} below.
For simplicity, we assume $\max(2,n)\le \theta(n) \ll n$, as in \cite{PTW}. 
Let $P^+(n)$ denote the largest prime factor of $n$ and put $P^+(1)=1$.
For a given integer $m$, which we call a \emph{starter}, 
let $\mathcal{A}_m$ be the set of all integers of the form $mp_1 p_2 \ldots p_k$, $P^+(m)<p_1 < \ldots < p_k$, 
which satisfy $p_i \le \theta(mp_1\ldots p_{i-1})$ for all $1\le i \le k$. Theorem 3.1 of \cite{PTW} states that
the counting function of $\mathcal{A}_m$ satisfies
\begin{equation}\label{Bmasymp}
A_{m}(x) = \frac{c_m x}{\log x} + O\left(\frac{\log^6(2m) x}{m \log^2 x}\right) \qquad (m\ge 1, x\ge 2),
\end{equation}
for some constant $c_m$. 

Let $\mathcal{S}$ be a set 
of natural numbers (starters) with the property that $\mathcal{A}_{m_1} \cap \mathcal{A}_{m_2} = \emptyset$
for all $m_1\neq m_2 \in \mathcal{S}$, and 
$ \sum_{m \in \mathcal{S}} m^{-1}\log^6 m \ll 1.$
Let $\mathcal{A}=\bigcup_{m\in \mathcal{S}} \mathcal{A}_m$ and assume that its counting function satisfies $A(x) \ll x/\log x$. 
As in \cite{PTW}, summing \eqref{Bmasymp} over $m\in \mathcal{S}$ yields
\begin{equation}\label{Aasymp}
 A(x)=\frac{C x}{\log x}+O\left( \frac{1}{\log^2 x}  \right),
\end{equation}
where $C=\sum_{m\in \mathcal{S}} c_m.$

\begin{theorem}\label{thm2}
The constant $C$ in \eqref{Aasymp} is given by
$$
C = \frac{6/\pi^2}{1-e^{-\gamma}} \sum_{m\in \mathcal{S}} \sum_{n\in \mathcal{A}_m} \frac{1}{n}  \Biggl( \sum_{P^+(m) < p \le \theta(n)}\frac{\log p}{p+1} \ - \ 
\log\left(\frac{n}{m}\right) \Biggr) \prod_{p\le \theta(n)} \left(1+\frac{1}{p}\right)^{-1}.
$$
\end{theorem}
For the set of $\varphi$-practical numbers, the set of starters $\mathcal{S}$ 
will be described in Section \ref{secphipra}, while $\theta(n)=n+2$. Indeed, given $m \in \mathcal{S}$, 
the integer  $mp_1 p_2 \ldots p_k$ with $P^+(m)<p_1 < \ldots < p_k$
is $\varphi$-practical if and only if $p_i \le 2+ mp_1\ldots p_{i-1}$ for all $1\le i \le k$, by \cite[Lemmas 3.3 and 4.1]{Thom}.

\begin{corollary}\label{Phicor2}
The constant $C$ in \eqref{Phiasymp} satisfies $0.945 < C < 0.967$.
\end{corollary}
Corollary \ref{Phicor2} is consistent with the empirical estimate $C\approx 0.96$ given in \cite[Section 6]{PTW}, which is
based on values of $P_\varphi(2^k)$ for $k\le 34$ and nonlinear regression.

\subsection{Squarefree analogues}

Let $\mathcal{D}^*_t$ denote the set of squarefree integers with $t$-dense divisors and 
let $D^*(x,t)$  be its counting function. 
Saias \cite[Theorem 1]{Saias1} showed that both $D(x,t)$ and $D^*(x,t)$ have order of magnitude $x \log t /\log x$, for $x\ge t \ge 2$.
The asymptotic estimate for $D^*(x,t)$, although not stated explicitly in the literature, is a special case of 
 \eqref{Bmasymp} (i.e. \cite[Thm. 3.1]{PTW}). With $\theta(n)=tn$ and $m=1$, we have 
\begin{equation}\label{Dasymp*}
D^*(x,t) = \frac{ c_t^*x}{\log x} \left\{1+O_t\left( \frac{1}{\log x}  \right)\right\},
\end{equation}
for some positive constant $ c_t^*$.
Theorem \ref{thm2} with $\mathcal{S}=\{1\}$ and $\mathcal{A}=\mathcal{A}_1 =\mathcal{D}^*_t$ yields
\begin{equation}\label{ct*def}
c^*_t = \frac{6/\pi^2}{1-e^{-\gamma}} \sum_{n\in \mathcal{D}^*_t} \frac{1}{n}  \Biggl( \sum_{ p \le tn}\frac{\log p}{p+1} \ - \ 
\log n \Biggr) \prod_{p\le t n} \left(1+\frac{1}{p}\right)^{-1}.
\end{equation}

\begin{corollary}\label{Dcors}
Table \ref{table3} shows values of the factor $c^*_t$ appearing in \eqref{Dasymp*}.
\end{corollary}

\begin{table}[h]
  \begin{tabular}{ | c | c | }
    \hline
    $\ t \ $ & $c^*_t$ \\ \hline
    $2$ &  0.06864.. \\ \hline
     $e$ &  0.1495... \\ \hline
    $3$ &  0.2618... \\ \hline
    $4$ & 0.4001... \\ \hline
      $5$ &  0.5898...   \\ \hline
     \end{tabular}
     \quad
     \begin{tabular}{ | c | c | }
    \hline
    $\ t \ $ & $c^*_t$ \\ \hline   
     $6$ & 0.7142... \\ \hline
    $7$ & 0.923.... \\ \hline  % likely 0.9238...
    $8$ &  1.0065... \\ \hline
     $9$ &  1.0978... \\ \hline
     $10$ & 1.1868... \\ \hline
      \end{tabular}
       \quad
     \begin{tabular}{ | c | c | }
    \hline
    $\ t \ $ & $c^*_t$ \\ \hline   
     $20$ & 2.017... \\ \hline
    $40$ & 2.854... \\ \hline
    $60$ &  3.389... \\ \hline
     $80$ &  3.778... \\ \hline
     $100$ & 4.066... \\ \hline
      \end{tabular}
       \quad
     \begin{tabular}{ | c | c | }
    \hline
    $\ t \ $ & $c^*_t$ \\ \hline   
     $10^3$ & 7.208... \\ \hline
    $10^4$ & 10.390... \\ \hline
    $10^5$ &  13.580... \\ \hline
     $10^6$ &  16.771... \\ \hline
     $10^7$ & 19.963... \\ \hline
      \end{tabular}
          \medskip
     \caption{Truncated values of $c^*_t$ derived from \eqref{ct*def}.}\label{table3}
\end{table}

The squarefree analogue of Corollary \ref{Dcor2} is as follows.
\begin{corollary}\label{Dcor2*}
Let $c^*_t$  be the factor in \eqref{Dasymp*}. Define $\delta^*_t$ implicitly by
\begin{equation}\label{cteq*}
c^*_t = \frac{\log t -\gamma -h +\delta^*_t}{(1-e^{-\gamma})\pi^2/6},
\end{equation}
where
\begin{equation}\label{hdef}
h=\sum_{p\ge 2} \frac{2\log p}{p^2-1}=1.139921...
\end{equation}
We have 
$ \delta^*_t  \ll \exp(-\sqrt{\log t})$ and 
\begin{equation}\label{deltat*}
| \delta^*_t |  \le  \frac{0.084}{\log^2 t}\qquad (t \ge 2^{25}).
\end{equation}
Assuming the Riemann hypothesis, we have 
$|\delta^*_t| \le \frac{\log^2 t}{7\sqrt{t}}$ for $t\ge 55$.
\end{corollary}
\begin{table}[h]
     \begin{tabular}{ | c | l | }
    \hline
    $\ t \ $ & $c^*_t$ \\ \hline   
    $10^8$ & 23.15...  \\ \hline
    $10^9$ & 26.34...  \\ \hline
    $10^{10}$ & 29.53...  \\ \hline
     \end{tabular}
               \quad
     \begin{tabular}{ | c | l | }
    \hline
    $\ t \ $ & $c^*_t$ \\ \hline   
    $10^{20}$ & 61.458...  \\ \hline
    $10^{30}$ & 93.378...  \\ \hline
    $10^{40}$ & 125.2980...  \\ \hline
             \end{tabular}
               \quad
     \begin{tabular}{ | c | l | }
    \hline
    $\ t \ $ & $c^*_t$ \\ \hline   
    $10^{60}$ & 189.1372...  \\ \hline
    $10^{80}$ & 252.9764...  \\ \hline
    $10^{100}$ & 316.8156...  \\ \hline
             \end{tabular}
   \medskip
     \caption{Truncated values of $c^*_t$ derived from \eqref{cteq*} and \eqref{deltat*}.}\label{table2*}
\end{table}

We briefly mention two other squarefree analogues. The estimate \eqref{Bmasymp} and Theorem \ref{thm2}, with $\theta(n)=n+2$ and $\mathcal{S}=\{1\}$,
give the asymptotic estimate and the constant factor for the count of squarefree $\varphi$-practical numbers.
For the count of squarefree practical numbers, one would first derive \eqref{Bmasymp} 
under the condition $\theta(n)\ll n \log \log n$, which introduces an extra factor of $\log \log x $ in the error term. 
Theorem \ref{thm2} then gives the constant factor with $\theta(n)=\sigma(n)+1$ and $\mathcal{S}=\{1\}$.

\subsection{Polynomial divisors over finite fields}
Let  $\mathbb{F}_q$  be the finite field with $q$ elements.
Let $f_q(n,m)$ be the proportion of polynomials $F$ of degree $n$ over  $\mathbb{F}_q$, with the property that the set of degrees of divisors of $F$ has no gaps of size greater than $m$.  
For example, $f_q(n,1)$ is the proportion of polynomials of degree $n$ over  $\mathbb{F}_q$ which have a divisor of every degree up to $n$.
Corollary 1 of \cite{DPD} states that,
uniformly for $q\ge 2$, $n\ge m\ge 1$, we have
\begin{equation}\label{Polyasymp}
f_q(n,m)=  \frac{c_q(m) m}{n+m} \left\{ 1+O\left(\frac{1}{n}+\frac{m^2}{n^2}\right)\right\},
\end{equation}
where $ 0< c_q(m)=(1-e^{-\gamma})^{-1} + O\left(m^{-1} q^{ -(m+1)\tau }\right)$ and $\tau = 0.205466...$. 
The estimate \eqref{Polyasymp} can be viewed as the polynomial analogue of \eqref{Dasymp2}.
By adapting the proof of Theorem \ref{thm1} to polynomials over finite fields, we obtain an expression for the factor $c_q(m)$.
\begin{theorem}\label{thm3}
The factor $c_q(m)$ in \eqref{Polyasymp} is given by
$$
c_q(m) = 
\frac{1/m}{1-e^{-\gamma}}\sum_{n\ge 0} f_q(n,m)  \Biggl(\sum_{k=1}^{n+m}\frac{k I_k}{q^k-1}\ -\ n\Biggr) \prod_{k=1}^{n+m} \left(1-\frac{1}{q^k}\right)^{I_k},
$$
where $I_k$ is the number of monic irreducible polynomials of degree $k$ over  $\mathbb{F}_q$.
\end{theorem}
\begin{corollary}\label{Polycor}
Table \ref{table4} shows values of the factor $c_q(m)$ appearing in \eqref{Polyasymp}.
\end{corollary}
\begin{table}[h]
  \begin{tabular}{ | c | c |c|c|c|c| }
    \hline
    $c_q(m)$ & $ m=1$ & m=2 & m=3 & m=4 & m=5 \\ \hline  
    $q=2$ & 3.400335... & 2.604818... & 2.412402... & 2.339007... & 2.310509...\\ \hline
    $q=3$ & 2.801735... & 2.388729... & 2.315222... & 2.291615... & 2.285304... \\ \hline
    $q=4$ & 2.613499... & 2.334793... & 2.295617... & 2.284202... & 2.281909... \\ \hline
    $q=5$ & 2.523222... & 2.313164... & 2.288755... & 2.282066... & 2.280999... \\ \hline
    $q=7$ & 2.436571... & 2.296082... & 2.283947... & 2.280853... & 2.280507... \\ \hline
    $q=8$ & 2.412648... & 2.292175... & 2.282950... & 2.280650... & 2.280428... \\ \hline
    $q=9$ & 2.394991... & 2.289561... & 2.282310... & 2.280534... & 2.280383... \\ \hline
     \end{tabular}
     \medskip
     
     \caption{Truncated values of $c_q(m)$.}\label{table4}
\end{table}

For example,  the proportion of polynomials of degree $n$ over  $\mathbb{F}_2$, which have a divisor of every degree up to $n$, is given by 
$$\frac{ 3.400335... }{ n} \left\{1+O\left(\frac{1}{n}\right)\right\}.$$
Theorem \ref{thm3} leads to an improvement of the asymptotic estimate for $c_q(m)$ mentioned below \eqref{Polyasymp}.
\begin{corollary}\label{Polycor2}
Uniformly for $q\ge 2$, $m\ge 1$, we have
$$ c_q(m) = \frac{1}{1-e^{-\gamma}} + O\left(\frac{1}{m^2 q^{(m+1)/2}}\right).$$
\end{corollary}
Combining Corollary \ref{Polycor2} with \eqref{Polyasymp}, we obtain the following improvement of \cite[Corollary 2]{DPD}.
Corollary \ref{Polycor3} is the polynomial analogue of Corollay \ref{Dasymp3}.
\begin{corollary}\label{Polycor3}
Uniformly for $q\ge 2$, $n\ge m\ge 1$, we have
\begin{equation*}
f_q(n,m)= \frac{m}{(1-e^{-\gamma})(n+m)} \left\{ 1+O\left(\frac{1}{n}+\frac{m^2}{n^2} +\frac{1}{m^2 q^{(m+1)/2}} \right)\right\}.
\end{equation*}
\end{corollary}

\bigskip

\section{Proof of Theorem \ref{thm1}}

Let $\chi(n)$ be the characteristic function of the set $\mathcal{B}_\theta$. 
Theorem 1 of \cite{SPA} shows that
\begin{equation}\label{S1}
 1 = \sum_{n\ge 1} \frac{\chi(n)}{n} \prod_{p\le \theta(n)} \left(1-\frac{1}{p}\right)
\end{equation}
if and only if $B(x)=o(x)$. 
Lemma \ref{lem1} extends this to an identity involving Dirichlet series for  $\re(s)>1$,
valid without any conditions on $\theta$ or $B(x)$.
\begin{lemma}\label{lem1}
For $\re(s)>1$ we have
\begin{equation}\label{lem1eq}
  1 = \sum_{n\ge 1} \frac{\chi(n)}{n^s} \prod_{p\le \theta(n)} \left(1-\frac{1}{p^s}\right).
\end{equation}
\end{lemma}

\begin{proof}
Let $P^-(n)$ denote the smallest prime factor of $n$ and put $P^-(1)=\infty$.
 Each natural number $m=p_1^{\alpha_1}p_2^{\alpha_2} \cdots p_k^{\alpha_k}$, $p_1<p_2<\ldots < p_k$, factors uniquely as $m=nr$, where 
 $n =p_1^{\alpha_1}p_2^{\alpha_2}\cdots p_j^{\alpha_j} \in \mathcal{B}_\theta$ and $P^-(r)=p_{j+1}>\theta(n)$. It follows that, for $\re(s)>1$,
 $$ \zeta(s)=\sum_{m\ge 1}\frac{1}{m^s} = 
  \sum_{n\ge 1} \frac{\chi(n)}{n^s} \prod_{p> \theta(n)} \left(1-\frac{1}{p^s}\right)^{-1}.$$
  Dividing by $\zeta(s)= \prod_{p\ge 2} \left(1-p^{-s}\right)^{-1}$ yields the result.
\end{proof}

\begin{lemma}\label{lem2}
For $\re(s)>1$ we have
\begin{equation}\label{lem2eq}
0 = \sum_{n\ge 1} \frac{\chi(n)}{n^s} \Biggl( \sum_{p\le \theta(n)}\frac{\log p}{p^s-1} - \log n\Biggr) 
\prod_{p\le \theta(n)} \left(1-\frac{1}{p^s}\right).
\end{equation}
\end{lemma}
\begin{proof}
Differentiate \eqref{lem1eq} with respect to $s$.
\end{proof}

While \eqref{S1} shows that \eqref{lem1eq} remains valid at $s=1$ if $B(x)=o(x)$, \eqref{lem2eq} does not hold at 
$s=1$. To see this, note that each term on the right-hand side of \eqref{lem2eq} is non-negative if $s=1$ and $\theta(n)=tn$, where
$t$ is a sufficiently large constant.
Define
$$ \alpha =  \sum_{n\ge 1} \frac{\chi(n)}{n} \Biggl( \sum_{p\le \theta(n)}\frac{\log p}{p-1} - \log n\Biggr) 
\prod_{p\le \theta(n)} \left(1-\frac{1}{p}\right),
$$
$$ F_N(s) =  \sum_{1\le n \le N} \frac{\chi(n)}{n^s} \Biggl( \sum_{p\le \theta(n)}\frac{\log p}{p^s-1} - \log n\Biggr) 
\prod_{p\le \theta(n)} \left(1-\frac{1}{p^s}\right),
$$
$$
 G_N(s) =  \sum_{ n > N} \frac{\chi(n)}{n^s} \Biggl( \log n - \sum_{p\le \theta(n)}\frac{\log p}{p^s-1}\Biggr) 
\prod_{p\le \theta(n)} \left(1-\frac{1}{p^s}\right),
$$
and let
$s_N = 1 + 1/\log^2 N$ for $N\ge 2$.
We have $F_N(s_N) = G_N(s_N)$ by Lemma \ref{lem2}, $\lim_{N\to \infty}   F_N(s_N) = \alpha$
by Lemma \ref{lem3}, and $\lim_{N\to \infty}   G_N(s_N) = (1-e^{-\gamma})c_\theta $ by Lemma \ref{lem4}.
Thus $\alpha = (1-e^{-\gamma})c_\theta$, which establishes Theorem \ref{thm1}.
It remains to prove Lemmas \ref{lem3} and \ref{lem4}, where we will assume
\begin{equation}\label{theta2}
n\le \theta(n) \ll n \log 2n (\log \log 3n)^b
\end{equation}
for some constant $b<-1$.

\begin{lemma}\label{lem3}
If $\theta$ satisfies \eqref{theta2}, $\displaystyle \lim_{N\to \infty}   F_N\left(1 + 1/\log^2 N\right) = \alpha$.
\end{lemma}

\begin{proof}
Let $s=s_N=1+1/\log^2 N$ and write
$$ |F_N(s)-\alpha | \le |F_N(s) - F_N(1)| + |F_N(1)-\alpha| = E_1+E_2,$$
say. Since $B(x) \ll x/\log x$ and 
$\log n \le \log\theta(n) \le \log n + O(\log\log n)$,
$$ E_2 \ll   \sum_{n> N} \frac{\chi(n)}{n\log \theta(n)} \Bigl| \log\theta(n) + O(1) - \log n\Bigr|
\ll  \sum_{n\ge N} \frac{\log \log n }{n \log^2 n} \ll \frac{\log \log N}{\log N}.$$
To estimate $E_1$, note that for $n\le N$,
$$n^{-s}=n^{-1}(1+O((s-1)\log n)) = n^{-1}(1+O(1/\log N)).$$ 
Similarly, $p^s-1=(p-1)(1+O((s-1)\log p))$, so that
$$  \sum_{p\le \theta(n)}\frac{\log p}{p^s-1}=O((s-1)\log^2 n) +   \sum_{p\le \theta(n)}\frac{\log p}{p-1}. $$
By the mean value theorem, there is an $\tilde{s}$ with $1< \tilde{s}<s$ such that
\begin{multline}\label{lem3mvt}
0< \prod_{p\le \theta(n)} \left(1-\frac{1}{p^s}\right) - \prod_{p\le \theta(n)} \left(1-\frac{1}{p}\right)
=(s-1) \sum_{p\le \theta(n)}\frac{\log p}{p^{\tilde{s}}-1} \prod_{p\le \theta(n)} \left(1-\frac{1}{p^{\tilde{s}}}\right) \\
\ll (s-1)\log(\theta(n)) \prod_{p\le \theta(n)} \left(1-\frac{1}{p^s}\right) \ll (1/\log N)  \prod_{p\le \theta(n)} \left(1-\frac{1}{p^s}\right),
\end{multline}
for $n\le N$. These estimates show that 
\begin{multline*}
 F_N(s) = \sum_{1 \le n\le N} \frac{\chi(n) }{n} \left(1+O\left(\frac{1}{\log N}\right)\right)   \\
\times \Biggl( O\left(\frac{\log^2 n}{\log^2 N}\right) +\sum_{p\le \theta(n)}\frac{\log p}{p-1}   - \log n\Biggr) 
\prod_{p\le \theta(n)} \left(1-\frac{1}{p}\right).
\end{multline*}
The contribution to the last sum from each of the two error terms is $\ll 1/\log N$. 
Hence $E_1\ll 1/\log N$ and the proof of Lemma \ref{lem3} is complete.
\end{proof}

\begin{lemma}\label{lem4}
If $\theta$ satisfies \eqref{theta2}, $\displaystyle \lim_{N\to \infty}   G_N\left(1 + 1/\log^{2} N\right) = (1-e^{-\gamma})c_\theta $.
\end{lemma}
\begin{proof}
Let $s=s_N=1+1/\log^2 N$ and write $I(y)=\int_0^y (1-e^{-t})\frac{dt}{t}$. Lemma 9.1 of \cite{Ten} shows that
\begin{equation}\label{lem4eq1}
\begin{split}
\prod_{p\le \theta(n)} \left(1-\frac{1}{p^s}\right) & =
\frac{ \exp\bigl\{-\gamma + I((s-1)\log \theta(n)) \bigr\}}{\log \theta(n)} \left(1+O\left(\frac{1}{\log \theta(n)}\right)\right) \\
 & = \frac{\exp\bigl\{-\gamma + I((s-1)\log n) \bigr\}}{\log n} \left(1+O\left(\frac{\log \log n}{\log n}\right)\right) ,
\end{split}
\end{equation}
by \eqref{theta2}.
By the prime number theorem,
\begin{equation}\label{lem4eq2}
\begin{split}
 \sum_{p\le \theta(n)}\frac{\log p}{p^s-1} & =\sum_{p\le  n}\frac{\log p}{p^s-1} + \sum_{n< p\le  \theta(n)}\frac{\log p}{p^s-1}\\
 & = O(1) + \frac{1-n^{1-s}}{s-1} +O\left(1+ \log ( \theta(n)/n) \right),
\end{split}
\end{equation}
for $n>N$.
The details behind the estimate for the sum over $p \le n$ are explained in \cite[Ex.1 of Sec.III.5]{Ten}. 
For the sum over $n<p \le \theta(n)$, note that the terms are $\ll \log(p) /p$. 
With these two estimates we have
\begin{multline*}
G_N(s)= \sum_{ n > N} \frac{\chi(n)}{n^s} \left( \log n -\frac{1-n^{1-s}}{s-1} +  O\left(\log \log n\right) \right)  \\
\times \frac{\exp\bigl\{-\gamma+I((s-1)\log n)\bigr\}}{\log n}  \left(1+O\left(\frac{\log \log n}{\log n}\right)\right) .
\end{multline*}
Since $e^{I(y)} \ll 1+y$ and $B(x)\ll x/\log x$, the contribution 
to the last sum from each of the two error terms is $\ll \log \log N /\log N$. 
Abel summation and the asymptotic estimate \eqref{Basymp} show that
\begin{multline*}
G_N(s)=o(1)  \\ + \int_{N}^{\infty} \frac{c_\theta}{y^s \log y} \left( \log y -\frac{1-y^{1-s}}{s-1}  \right) 
\frac{\exp\bigl\{-\gamma+ I((s-1)\log y)\bigr\}}{\log y}dy,
\end{multline*}
as $N\to \infty$.
With the change of variables $u=(s-1)\log y$, this simplifies to
$$
G_N(s) = o(1)+ e^{-\gamma}  c_\theta  \int_{1/\log N}^{\infty} \frac{u-1+e^{-u}}{u^2 e^u } \exp\bigl(I(u)\bigr) du.
$$
Note that the integrand is equal to $((I'(u))^2 +I''(u)) \exp(I(u))$, so that an antiderivative is $I'(u) \exp(I(u))$.
Thus the last integral equals
$$
\lim_{u\to \infty} I'(u) \exp(I(u)) - I'(1/\log N) \exp(I(1/\log N))=e^\gamma -1 +o(1),
$$
as $N\to \infty$, since $I(u)=\gamma + \log u + \int_u^\infty e^{-t} t^{-1} dt$ by \cite[Ex.1 of Sec.7.2.1]{MV}.
\end{proof}

\section{Proof of Theorem \ref{thm2}}

The proof of Theorem \ref{thm2} closely follows that of Theorem \ref{thm1}.
\begin{lemma}\label{lemA1}
For $m\ge 1$ and $\re(s)>1$ we have
\begin{equation}\label{lemA1eq}
 \frac{1}{m^s}\prod_{p\le P^+(m)} \left(1+\frac{1}{p^s}\right)^{-1}  
  = \sum_{n \in \mathcal{A}_m} \frac{1}{n^s} \prod_{p\le \theta(n)} \left(1+\frac{1}{p^s}\right)^{-1}.
\end{equation}
\end{lemma}

\begin{proof}
 Each natural number of the form $m p_1 \ldots p_k$, $P^+(m)<p_1< \ldots < p_k$, factors uniquely as
 $nr$, where 
 $n =m p_1 p_2 \cdots p_j \in \mathcal{A}_m$ and $P^-(r)=p_{j+1}>\theta(n)$. Thus, for $\re(s)>1$,
 $$ 
 \frac{1}{m^s}\prod_{p> P^+(m)} \left(1+\frac{1}{p^s}\right)  
  = \sum_{n \in \mathcal{A}_m} \frac{1}{n^s} \prod_{p> \theta(n)} \left(1+\frac{1}{p^s}\right).
  $$
 The result follows from dividing by $\prod_{p\ge 2} \left(1+1/p^s\right)$.
\end{proof}

\begin{lemma}\label{lemA2}
For $m\ge 1$ and $\re(s)>1$ we have
\begin{equation}\label{lemA2eq}
0= \sum_{n \in \mathcal{A}_m} \frac{1}{n^s} 
  \left( \sum_{P^+(m) < p \le \theta(n) } \frac{\log p}{p^s +1} - \log\left(\frac{ n}{m}\right) \right) \prod_{p\le \theta(n)} \left(1+\frac{1}{p^s}\right)^{-1}.
\end{equation}
\end{lemma}
\begin{proof}
Differentiating \eqref{lemA1eq} with respect to $s$ shows that
\begin{multline*}
 \frac{1}{m^s} 
  \left( \sum_{p\le P^+(m) } \frac{\log p}{p^s +1} - \log m\right) \prod_{p\le P^+(m)} \left(1+\frac{1}{p^s}\right)^{-1}\\
  = \sum_{n \in \mathcal{A}_m} \frac{1}{n^s} 
  \left( \sum_{p\le \theta(n) } \frac{\log p}{p^s +1} - \log n\right) \prod_{p\le \theta(n)} \left(1+\frac{1}{p^s}\right)^{-1}.
\end{multline*}
The result now follows from Lemma \ref{lemA1}.
\end{proof}

Define
$$ \alpha_m = \sum_{n \in \mathcal{A}_m} \frac{1}{n} 
  \left( \sum_{P^+(m) < p \le \theta(n) } \frac{\log p}{p +1} -\log\left(\frac{ n}{m}\right) \right) \prod_{p\le \theta(n)} \left(1+\frac{1}{p}\right)^{-1} ,
$$
$$ F_{m,N}(s) =  \sum_{\substack{n \in \mathcal{A}_m \\n\le N }} \frac{1}{n^s} 
  \left( \sum_{P^+(m) < p \le \theta(n) } \frac{\log p}{p^s +1} - \log\left(\frac{ n}{m}\right) \right) \prod_{p\le \theta(n)} \left(1+\frac{1}{p^s}\right)^{-1},
$$
$$
 G_{m,N}(s) = \sum_{\substack{ n \in \mathcal{A}_m \\n> N}} \frac{1}{n^s} 
  \left( \log\left(\frac{ n}{m}\right) - \sum_{P^+(m) < p \le \theta(n) } \frac{\log p}{p^s +1} \right) \prod_{p\le \theta(n)} \left(1+\frac{1}{p^s}\right)^{-1},
$$
and let
$s_N = 1 + 1/\log^2 N$ for $N\ge 2$.
We have $F_{m,N}(s_N) = G_{m,N}(s_N)$ by Lemma \ref{lemA2}, $\lim_{N\to \infty}   F_{m,N}(s_N) = \alpha_m$
by Lemma \ref{lemA3}, and $\lim_{N\to \infty}   G_{m,N}(s_N) =\zeta(2) (1-e^{-\gamma})c_m $ by Lemma \ref{lemA4},
where $c_m$ is the constant in \eqref{Bmasymp}.
Thus $\alpha_m =\zeta(2) (1-e^{-\gamma})c_m$. Theorem \ref{thm2} now follows from summing over $m\in \mathcal{S}$.
The proofs of Lemmas \ref{lemA3} and \ref{lemA4} are almost identical to those of Lemmas \ref{lem3} and \ref{lem4}.

\begin{lemma}\label{lemA3}
Let $m\ge 1$ be fixed and assume $\theta$ satisfies \eqref{theta2}. Then 
 $$ \lim_{N\to \infty}   F_{m,N}\left(1 + 1/\log^2 N\right) = \alpha_m.$$
\end{lemma}

\begin{lemma}\label{lemA4}
Let $m\ge 1$ be fixed and assume $\theta$ satisfies \eqref{theta2}. Then 
$$\displaystyle \lim_{N\to \infty}   G_{m,N}\left(1 + 1/\log^{2} N\right) = \zeta(2) (1-e^{-\gamma})c_m .$$
\end{lemma}

\section{Proof of Theorem \ref{thm3}}

The proof of Theorem \ref{thm3} is analogous to that of Theorem \ref{thm1},
with power series replacing Dirichlet series.
\begin{lemma}\label{lemF1}
For $m\ge 1$ and $|z|<1$ we have
\begin{equation}\label{lemF1eq}
 1  = \sum_{n \ge 0 } f_q(n,m) \, z^n \prod_{k=1}^{n+m} \left(1-\frac{z^k}{q^k}\right)^{I_k}.
\end{equation}
\end{lemma}

\begin{proof}
Lemma 5 of \cite{DPD} implies that 
$$ z^j= \sum_{n=0}^j f_q(n,m) z^n \, r_q(j-n,n+m) z^{j-n}\qquad (j\ge 0, m\ge 0),$$
where $r_q(n,m)$ denotes the proportion of polynomials of degree $n$ over $\mathbb{F}_q$,
all of whose non-constant divisors have degree $>m$.
Summing over $j\ge 0$ yields
$$
\frac{1}{1-z} = \sum_{n\ge 0} f_q(n,m) z^n \sum_{j\ge n} r_q(j-n,n+m) z^{j-n},
$$
for $|z|<1$.
The inner sum equals
$$ \sum_{j\ge 0} r_q(j,n+m) z^{j} = \prod_{k>n+m}\left(1-\frac{z^k}{q^k}\right)^{-I_k}
=\frac{1}{1-z}\prod_{k=1}^{n+m}\left(1-\frac{z^k}{q^k}\right)^{I_k}.
$$
The result now follows from multiplying by $(1-z)$.
\end{proof}

\begin{lemma}\label{lemF2}
For $m\ge 1$ and $|z|<1$ we have
\begin{equation}\label{lemF2eq}
0=  \sum_{n \ge 0 } f_q(n,m) \, z^{n-1} \left(n- \sum_{k=1}^{n+m} \frac{k I_k}{(q/z)^k -1} \right)
\prod_{k=1}^{n+m} \left(1-\frac{z^k}{q^k}\right)^{I_k}.
\end{equation}
\end{lemma}
\begin{proof}
Differentiate \eqref{lemF1eq} with respect to $z$.
\end{proof}

Define
$$ \alpha_{q,m} = \sum_{n \ge 0 } f_q(n,m)\left( \sum_{k=1}^{n+m} \frac{k I_k}{q^k -1} \ -n\right)
\prod_{k=1}^{n+m} \left(1-\frac{1}{q^k}\right)^{I_k},
$$
$$ F_{q,m,N}(z) =  \sum_{n=0}^N f_q(n,m) \, z^{n-1} \left( \sum_{k=1}^{n+m} \frac{k I_k}{(q/z)^k -1} \ -n \right)
\prod_{k=1}^{n+m} \left(1-\frac{z^k}{q^k}\right)^{I_k},
$$
$$
 G_{q,m,N}(z) =   \sum_{n>N} f_q(n,m) \, z^{n-1} \left(n - \sum_{k=1}^{n+m} \frac{k I_k}{(q/z)^k -1}  \right)
\prod_{k=1}^{n+m} \left(1-\frac{z^k}{q^k}\right)^{I_k},
$$
and let
$z_N = \exp(-1/N^2)$ for $N\ge 1$.
We have $F_{q,m,N}(z_N) = G_{q,m,N}(z_N)$ by Lemma \ref{lemF2}, $\lim_{N\to \infty}   F_{q,m,N}(z_N) = \alpha_{q,m}$
by Lemma \ref{lemF3}, and $\lim_{N\to \infty}   G_{q,m,N}(z_N) =(1-e^{-\gamma})m c_q(m) $ by Lemma \ref{lemF4},
where $c_q(m)$ is the constant in \eqref{Polyasymp}.
Thus $\alpha_{q,m} =(1-e^{-\gamma}) m c_q(m)$, which is what we need to show.

\begin{lemma}\label{lemF3}
For $m\ge 1$ we have
 $$ \lim_{N\to \infty}   F_{q,m,N}\left(\exp(- 1/N^2)\right) = \alpha_{q,m}.$$
\end{lemma}

\begin{lemma}\label{lemF4}
For $m\ge 1$ we have
$$\displaystyle \lim_{N\to \infty}   G_{q,m,N}\left(\exp(- 1/N^2)\right) = (1-e^{-\gamma}) m\, c_q(m) .$$
\end{lemma}

The proofs of Lemmas \ref{lemF3} and \ref{lemF4} are analogous to those of Lemmas \ref{lem3} and \ref{lem4},
with \eqref{In} playing the role of the prime number theorem. In particular, with  $z=z_N= \exp(-1/N^2)$, 
the analogue of \eqref{lem3mvt} is
$$
0<\prod_{k=1}^{n+m} \left(1-\frac{z^k}{q^k}\right)^{I_k}-\prod_{k=1}^{n+m} \left(1-\frac{1^k}{q^k}\right)^{I_k}
\ll \frac{n+m}{N^2} \prod_{k=1}^{n+m} \left(1-\frac{z^k}{q^k}\right)^{I_k},
$$
for $n\le N$, by the mean value theorem and \eqref{In}.
The analogue of \eqref{lem4eq1} is
$$ \prod_{k=1}^{n+m} \left(1-\frac{z^k}{q^k}\right)^{I_k} = \frac{\exp\{-\gamma +I((n+m)/N^2)\} }{n+m}
\left(1+O\left(\frac{1}{N}\right)\right), \quad (n>N), 
$$
which can be derived from \eqref{In}.
The estimate \eqref{lem4eq2} corresponds to
$$  \sum_{k=1}^{n+m} \frac{k I_k}{(q/z)^k -1} = \frac{1-z^{n+m}}{1-z}+O(1), \quad (n>N),$$
which follows from Lemma \ref{Lq} and \eqref{In}.

\section{Proofs of corollaries to Theorem \ref{thm1}}\label{CorThm1Sec}

We need to estimate $\alpha = (1-e^{-\gamma}) c_\theta  = \lim_{N\to \infty} \alpha_N,$ where 
 \begin{equation}\label{aNdef}
 \alpha_N =   \sum_{n\le N} \frac{\chi(n)}{n}\Delta(n)
\prod_{p\le \theta(n)} \left(1-\frac{1}{p}\right)
\end{equation}
and
$$ \Delta(n) =\sum_{p\le \theta(n)}\frac{\log p}{p-1} - \log n . $$
Assume that there are real numbers $L_N$ and $R_N$ such that 
\begin{equation}\label{Deltaest}
L_N < \Delta(n)<  R_N \qquad (n>N),
\end{equation}
and let
$$  \varepsilon_N = \sum_{n>N}  \frac{\chi(n)}{n}\prod_{p\le \theta(n)} \left(1-\frac{1}{p}\right) 
=1 - \sum_{n\le N}  \frac{\chi(n)}{n}\prod_{p\le \theta(n)} \left(1-\frac{1}{p}\right),$$
by \eqref{S1}.
The last equation allows us to calculate $\varepsilon_N$ on a computer based on values of $\chi(n)$ and $\theta(n)$ for $n\le N$.
We have
\begin{equation}\label{alphaest}
\alpha_N + L_N \varepsilon_N < \alpha  <\alpha_N + R_N \varepsilon_N.
 \end{equation}
 
To determine values for $L_N$ and $R_N$ which satisfy \eqref{Deltaest}, 
we need an effective estimate for the sum over primes in the definition of  $\Delta(n)$.

\begin{lemma}\label{pslem}
Let
\begin{equation}\label{etadef}
 \eta(x) = \sum_{p\le x} \frac{\log p}{p-1}  - \log x + \gamma .
\end{equation}
We have $\eta(x) \ll \exp(-\sqrt{\log x})$ and
$$ |\eta(x)| \le  E(x) := \frac{0.084}{\log^2 x}\qquad (x \ge 2^{25}).$$
Assuming the Riemann hypothesis, we have 
$ |\eta(x)| \le \frac{\log^2 x}{7\sqrt{x}}$ for $x\ge 25$.
\end{lemma}
\begin{proof}
Rosser and Schoenfeld \cite[Eq. 4.21]{RS} give the relation
$$\tilde{\eta}(x):= \eta(x) +\sum_{p>x} \frac{\log p}{p(p-1)} = \frac{\vartheta(x)-x}{x}-\int_x^\infty \frac{\vartheta(y)-y}{y^2} dy,$$
where $\vartheta(x)=\sum_{p\le x} \log p$. The estimate $ \eta(x)  \ll \exp(-\sqrt{\log x})$ now follows from the prime number theorem.

Axler \cite[Prop. 8]{Axler} shows that for $x\ge 30972320=2^{24.88...}$, 
$$| \tilde{\eta}(x)| \le  \frac{3}{40 \log^2 x}\left(1+\frac{2}{\log x}\right), $$
which implies our estimate $|\eta(x)|  \le  E(x)$ for $x\ge 2^{25}$, since
\begin{equation}\label{etatilde}
0< \sum_{p>x} \frac{\log p}{p(p-1)} < \int_x^\infty \frac{\log y}{y^2} dy = \frac{1+\log x}{x}. 
\end{equation}

Assuming the Riemann hypothesis, Schoenfeld \cite[Cor. 2]{Schoen} gives a bound for $|\tilde{\eta}(x)|$, 
which together with \eqref{etatilde} yields our bound for $|\eta(x)|$ if $x\ge 160000$. 
For $25\le x\le 160000$, we verify the result with a computer.
\end{proof}

\subsection{Proof of Corollaries \ref{Dcor} and \ref{Phicor}}\label{calcsection}
For Corollary \ref{Dcor} we have $\theta(n)=tn$ and
\begin{equation}\label{Deltatn}
  \Delta(n) = \sum_{p\le tn}\frac{\log p}{p-1} - \log n  = \eta(tn) +\log t - \gamma.
\end{equation}
Lemma \ref{pslem} shows that condition \eqref{Deltaest} is satisfied with $R_N= \log t -\gamma+E(t N)$ and $L_N= \log t -\gamma - E(tN)$, 
if $tN\ge 2^{25}$.
For $N=2^{25}$ and $t=2$, we calculate $\alpha_N$ and $\varepsilon_N$ with a computer 
and find that \eqref{alphaest} yields $1.224806  < c_2< 1.224852 $, hence $c_2 = 1.2248...$.
All the other estimates in Corollaries \ref{Dcor} and \ref{Phicor} are derived similarly. To obtain the decimal places as shown, $tN\le 2^{30}$ 
suffices in all cases.

\subsection{Proof of Corollary \ref{Dcor2}}
Theorem \ref{thm1} and \eqref{Deltatn} yield
$$(1-e^{-\gamma}) c_t =  \sum_{n\in\mathcal{B}} \frac{1}{n}\Bigl(\log t -\gamma + \eta(tn)\Bigr) \prod_{p\le nt} \left(1-\frac{1}{p}\right).$$
Together with \eqref{S1} we obtain
\begin{equation}\label{deltaeta}
\inf_{n\ge 1}\eta(nt)\le \delta_t \le \sup_{n\ge 1}\eta(nt).
\end{equation}

The other estimates for $\delta_t$ follow from \eqref{deltaeta} and Lemma \ref{pslem}, since $E(x)$ is decreasing for $x\ge 2$ and $\frac{\log^2 x}{7\sqrt{x}}$ is decreasing for $x\ge 55$.

\subsection{Proof of Corollary \ref{Pcor}}
We use the fact that $\theta(n)=\sigma(n)+1\ge 2n$ whenever $n$ is practical \cite[Lemma 2]{Mar}.
We have
$$ \Delta(n)\ge \sum_{p\le 2n} \frac{\log p}{p-1} - \log n = \eta(2n)+\log 2 - \gamma > \log 2 - \gamma -E(2n),$$
 for $2n\ge 2^{25}$, and hence
 $$ \alpha \ge \alpha_N + (\log 2 - \gamma - E(2N)) \varepsilon_N,$$
 for $2N \ge 2^{25}$. The lower bound in Corollary \ref{Pcor} now follows from calculating $\alpha_N$ and $\varepsilon_N$ for $N=2^{26}$.

For the upper bound in Corollary \ref{Pcor}, we have, for $n\ge 2^{25}$, 
\begin{equation}\label{rob}
\begin{split}
\Delta(n) & =\sum_{p\le \sigma(n)+1} \frac{\log p}{p-1} - \log n \\
& = \eta(\sigma(n)+1) +\log(\sigma(n)+1) -\gamma -\log n \\
& \le E(2n) + \log((\sigma(n)+1)/n) -\gamma \\
& \le  \log_3(n) + E(2n) +1/n +\frac{0.6483}{e^\gamma (\log_2 n)^2},\\
&=:  \log_3(n)  + \beta_n,
\end{split}
\end{equation}
say, where $\log_k$ denotes the $k$-fold logarithm.
For the last inequality of \eqref{rob} we used Robin's \cite[Theorem 2]{Robin} unconditional upper bound
$$
\frac{\sigma(n)}{n} \le e^\gamma \log_2 n\left(1+\frac{0.6483}{e^\gamma (\log_2 n)^2} \right) \qquad (n\ge 3)
$$ 
and the inequality $\log(1+x)\le x$.
If $n\ge 2$ is practical, 
$$
 \prod_{p\le \sigma(n)+1}\left(1-\frac{1}{p}\right) \le  \prod_{p\le 2n}\left(1-\frac{1}{p}\right)
 \le \frac{e^{-\gamma}}{\log 2n} \left(1+\frac{1}{2\log^2 (2n)}\right) <  \frac{e^{-\gamma}}{\log n},
$$
by \cite[Theorem 7]{RS}.
For $M>N\ge 2^{25}$, we get
\begin{equation*}
\begin{split}
  \alpha-\alpha_N & \le \varepsilon_N \beta_N 
   + \sum_{n>N} \frac{\chi(n)}{n} \log_3 n \prod_{p\le \sigma(n)+1}\left(1-\frac{1}{p}\right) \\
& \le \varepsilon_N \beta_N + \varepsilon_N  \log_3 M 
   + \sum_{n>M} \frac{\chi(n)}{n} (\log_3 n -\log_3 M) \prod_{p\le \sigma(n)+1}\left(1-\frac{1}{p}\right)  \\
  & \le  \varepsilon_N \beta_N + \varepsilon_N  \log_3 M 
    + e^{-\gamma} \sum_{n>M} \frac{\chi(n)}{n \log n} (\log_3 n -\log_3 M) \\
    & \le  \varepsilon_N \beta_N + \varepsilon_N  \log_3 M  + e^{-\gamma}\frac{a(1+1/\log M)}{b^2(\log M)^b \log_2 M},
 \end{split}
\end{equation*}
for $a=1.185$ and $b=\mu-\nu$, by Lemmas \ref{PUB} and \ref{DNlem}.
For $N=2^{26}$, the last expression is minimized when $\log_3 M=4.15$, which results 
in the upper bound  $c = \alpha/(1-e^{-\gamma}) <  1.693$.
If one could use $a=2$ and $b=1$, i.e. improve Lemma \ref{PUB} to $P(x) \le 2x/\log x$ for $x\ge 2^{26}$, which is likely true based on empirical evidence, 
the same method would yield $c< 1.441$.

\section{Proofs of Corollaries to Theorem \ref{thm2}}\label{secphipra}

Following \cite{PTW}, the set of $\varphi$-practical numbers arises as described in Section \ref{phiprasec} with $\theta(n)=n+2$ and
a set of starters $\mathcal{S}$ defined as follows.  
Let $P^+(n)$ (resp. $P^-(n)$) denote the largest (resp. smallest) prime factor of $n$.
We call $d$ an initial divisor of $n$ if $d|n$ and $P^+(d)<P^-(n/d)$.
A starter is a $\varphi$-practical number $m$ such that either
$m/P^+(m)$ is not $\varphi$-practical or $P^+(m)^2|m$. A $\varphi$-practical number $n$ is said
to have starter $m$ if $m$ is a starter, $m$ is an initial divisor of $n$, and $n/m$ is squarefree.
Each $\varphi$-practical number $n$ has a unique starter.

\subsection{The lower bound in Corollary \ref{Phicor2}}

Let $h$ be as in \eqref{hdef} and write
\begin{equation*}
  H(x)=\sum_{p\le x} \frac{\log p}{p+1}.
\end{equation*}
We have
$$H(x) +h+\gamma - \log x = \eta(x) + \sum_{p>x}\frac{2 \log p}{p^2-1} < \eta(x)+h .$$
Lemma \ref{pslem} implies
\begin{equation}\label{Hest}
 -E(x)< H(x) +h+\gamma - \log x < 2,
\end{equation}
where $x\ge 2^{25}$ in the first inequality, and $x\ge 1$ in the second.
We need a lower bound for $ C \zeta(2) (1-e^{-\gamma}) $, which by Theorem \ref{thm2} equals
$$ 
 \sum_{m\in \mathcal{S}} \sum_{n\in \mathcal{A}_m} \frac{1}{n}  
 \Biggl( \Bigl\{\log m -H(P^+(m)) - \lambda\Bigr\} +  \Bigl\{H(n+2) - \log n +\lambda\Bigr\} \Biggr) 
 \prod_{p\le n+2} \left(1+\frac{1}{p}\right)^{-1},
$$
for any real number $\lambda$. Lemma 3.5 of \cite{PTW} shows that \eqref{lemA1eq} is valid for $s=1$.
Thus the last expression can be written as
\begin{multline*}
\sum_{m\in \mathcal{S}} \frac{1}{m}  
 \Bigl\{\log m -H(P^+(m)) - \lambda\Bigr\} 
 \prod_{p\le P^+(m)} \left(1+\frac{1}{p}\right)^{-1} \\
+
\sum_{n\in \mathcal{A}} \frac{1}{n}  
 \Bigl\{H(n+2) - \log n +\lambda\Bigr\}
 \prod_{p\le n+2} \left(1+\frac{1}{p}\right)^{-1}\\
 = U(\lambda)+V(\lambda),
\end{multline*}
say. Let $U_N(\lambda)$ and $V_N(\lambda)$ denote the corresponding partial sums.
For a given $N$, we pick $\lambda$ such that the terms of both series are positive for $m,n>N$.
Then $C\zeta(2) (1-e^{-\gamma})= U(\lambda)+V(\lambda) > U_N(\lambda)+V_N(\lambda)$, which yields a lower bound for $C$
after dividing by $\zeta(2) (1-e^{-\gamma})$. For $n>N=2^{30}$, \eqref{Hest} implies
$$ H(n+2)-\log n \ge H(n)-\log n \ge -h -\gamma -E(n) > -1.7174.$$
For the series $U(\lambda)$, note that $m\in \mathcal{S}$ implies $m$ is $\varphi$-practical.
Thus $P^+(m) \le 2+ m/P^+(m)$, which yields $P^+(m)\le 2 +\sqrt{m}$. We have
$$ \log m -H(P^+(m)) \ge \log m -H(2+\sqrt{m}) \ge \log(m/(\sqrt{m}+2)) -1 \ge 2 ,$$
by \eqref{Hest}, for $m\ge 500$. Thus $\lambda=1.7174$ ensures that the terms in both series are positive for $m,n>N$. 
With $N=2^{30}$, we get $C> (U_N(\lambda)+V_N(\lambda))/(\zeta(2) (1-e^{-\gamma}))>0.945$

\subsection{The upper bound in Corollary \ref{Phicor2}}

Using a similar strategy as for the lower bound would require an explicit upper bound for the counting function of starters,
since $ \log m -H(P^+(m))$ grows unbounded.
Instead, we will define a function $\theta(n)$ such that $\mathcal{A}\subset \mathcal{B}_\theta$ and hence $C\le c_\theta$.
We then estimate $c_\theta$ as in Section \ref{CorThm1Sec}.

Let 
$$
\theta(n)=
\begin{cases}
n+2 & \text{if } n\in \mathcal{A}; \\
mp+2      &  \text{if } n=mp^a, \ m\in \mathcal{A}, \ p=m+2, \ a\ge2;\\
n+2-m\varphi(n/m) & \text{else,} 
\end{cases}
$$
where $m$ denotes the largest initial divisor of $n$ with $m\in \mathcal{A}$. 

To show that  $\mathcal{A}\subset \mathcal{B}_\theta$,  assume that $n\notin  \mathcal{B}_\theta$.
Then $n$ has an initial divisor $\tilde{n}$ such that $q:=P^-(n/\tilde{n})$ satisfies $q > \theta(\tilde{n})$.
First, if $\tilde{n}\in \mathcal{A}$, then $q>\tilde{n}+2$ and $n\notin \mathcal{A}$ by \cite[Lemma 3.3]{Thom}.
Second, if $\tilde{n}=mp^a$, $m\in \mathcal{A}$, $p=m+2$, $a\ge2$, then $q>mp+2$.
Since $\varphi(p^2)=p(p-1)>mp+1$ and $\varphi(q)=q-1>mp+1$, the number 
$mp+1$ cannot be written as a subsum of $\sum_{d|n} \varphi(d)$, so $n\notin \mathcal{A}$.
Third, if $m<\tilde{n}$ is the largest initial divisor of $\tilde{n}$ with $m\in \mathcal{A}$,
then $q>\tilde{n}+2-m\varphi(\tilde{n}/m)$, hence $\varphi(q)>\tilde{n}+1-m\varphi(\tilde{n}/m)$.
Lemmas 5.2 and 5.3 of \cite{PTW} show that the number $\tilde{n}+1-m\varphi(\tilde{n}/m)$ 
cannot be written as a subsum of $\sum_{d|n} \varphi(d)$, so $n\notin \mathcal{A}$.

To estimate $c_\theta$, we use Theorem \ref{thm1} and proceed as in Section \ref{CorThm1Sec}.
Since $\theta(n)\le n+2$, we have
\begin{equation*}
\begin{split}
 \Delta(n) \le \sum_{p\le n+2} \frac{\log p}{p-1} - \log n & = \eta(n+2)-\gamma +\log(n+2) -\log n \\
& \le  \frac{0.225}{\log^2 (n+2)} -\gamma + \frac{2}{n} =:R_n,
\end{split}
\end{equation*}
for $n\ge 2^{21}$, where the estimate for $\eta(n+2)$ follows from Lemma \ref{pslem} for $n\ge 2^{25}$, and for $2^{21}\le n < 2^{25}$ we
verify it by computation. For $N=2^{21}$, we get $c_\theta < (\alpha_N + \varepsilon_N R_N)/(1-e^{-\gamma}) < 0.967$.

\subsection{Proof of Corollaries \ref{Dcors} and \ref{Dcor2*}}
The calculations for Corollary \ref{Dcors} are analogous to those for Corollary \ref{Dcor}, with \eqref{Deltatn} replaced by 
$$
\Delta^*(n) = \sum_{p\le tn}\frac{\log p}{p+1} - \log n  = \log t - \gamma -h + \eta^*(tn) ,
$$
where $h$ is given by \eqref{hdef},
\begin{equation}\label{etadef*}
 \eta^*(x) = \eta(x)+\sum_{p> x} \frac{2\log p}{p^2-1},
\end{equation}
and $\eta(x)$ is as in \eqref{etadef}. 
Lemma 3.5 of \cite{PTW} shows that \eqref{lemA1eq} is valid for $s=1$, that is
 $$ 1=\sum_{n\in \mathcal{D}^*_t} \frac{1}{n} \prod_{p\le nt} \left(1+\frac{1}{p}\right)^{-1},$$
since $m=1$ and $\theta(n)=nt$. From \eqref{ct*def} we have
$$ c_t^* (1-e^{-\gamma}) \pi^2/6 = \sum_{n\in \mathcal{D}^*_t} \frac{1}{n} \Bigl(\log t - \gamma -h + \eta^*(tn) \Bigr)  \prod_{p\le nt} \left(1+\frac{1}{p}\right)^{-1},$$
which yields
\begin{equation}\label{deltaeta*}
\inf_{n\ge 1}\eta^*(nt)\le \delta^*_t \le \sup_{n\ge 1}\eta^*(nt).
\end{equation}
The other assertions follow from \eqref{deltaeta*}, because Lemma \ref{pslem} 
remains valid when $\eta(x)$ is replaced by $\eta^*(x)$,
if we replace the range $x\ge 25$ by $x\ge 33$ for the bound that assumes the Riemann hypothesis.
Indeed, we have 
$$ 0<\eta^*(x)-\tilde{\eta}(x) = \sum_{p>x} \frac{2\log p}{p^2-1}-\frac{\log p}{p(p-1)} < \sum_{p>x} \frac{\log p}{p^2-1} < \frac{1+\log x}{x},$$
the same upper bound as we used for $\tilde{\eta}(x)-\eta(x)$ in the proof of Lemma \ref{pslem}.

\section{Proofs of corollaries to Theorem \ref{thm3}}
Theorem \ref{thm3} says that $c_q(m)=\frac{\alpha/m}{1-e^{-\gamma}}$, where $\alpha = \lim_{N\to \infty} \alpha_N$,
$$\alpha_N = \sum_{0\le n\le N} f_q(n,m) \Delta_m(n) \lambda(n+m),$$
$$\lambda(n) = \prod_{k=1}^n \left(1-\frac{1}{q^k}\right)^{I_k} $$
and
$$ \Delta_m(n) = \sum_{k=1}^{n+m}\frac{k I_k}{q^k-1}\ -\ n.$$
We have
$$\Delta_m(n) = m+ \sum_{k=1}^{n+m}\frac{k I_k - (q^k-1)}{q^k-1}=  m +\sum_{k>n+m}\frac{q^k-1 - k I_k}{q^k-1} ,$$
by Lemma \ref{Lq}.
With the bounds \cite[p. 142, Ex. 3.26 and Ex. 3.27]{LN} 
\begin{equation}\label{In}
 \frac{q^k}{k} - \frac{2q^{k/2}}{k}< I_k \le \frac{q^k}{k} \qquad (k\ge 1), 
\end{equation}
we obtain
\begin{equation}\label{Deltam}
-L_{n+m}:=-\sum_{k>n+m}\frac{1}{q^k-1}\le  \Delta_m(n) -m \le  \sum_{k>n+m}\frac{2q^{k/2}-1}{q^k-1}=:R_{n+m}.
\end{equation}
Lemma 7 of \cite{DPD} shows that 
\begin{equation}\label{lem7DPD}
1=\sum_{n\ge 0} f_q(n,m)\lambda(n+m),
\end{equation}
so that
$$ \varepsilon_N:=\sum_{n>N} f_q(n,m)\lambda(n+m)
= 1-\sum_{0\le n\le N} f_q(n,m)\lambda(n+m),$$
which we can calculate on a computer. We have
$$  \alpha_N+(m- L_{N+m})\varepsilon_N  \le  \alpha \le \alpha_N+(m+ R_{N+m})\varepsilon_N,$$
which yields bounds for $c_q(m)$ upon dividing by $m(1-e^{-\gamma})$.
To obtain the accuracy as shown in Table \ref{table4}, $N=50$ or less suffices in all cases.

From \eqref{Deltam} and \eqref{lem7DPD} we get
\begin{equation*}
\begin{split}
 \alpha &=  \sum_{n\ge 0} f_q(n,m) \Delta_m(n) \lambda(n+m) \\
 & =  \sum_{n\ge 0} f_q(n,m)\left\{m+O\left(q^{-(n+m+1)/2}\right)\right\} \lambda(n+m)\\
 & = m + O\left(\frac{1}{m q^{(m+1)/2}} \right),
 \end{split}
 \end{equation*}
since $f_q(n,m)\ll m/(n+m)$ and $\lambda(n+m) \ll 1/(n+m)$. 
Dividing by $m(1-e^{-\gamma})$ yields Corollary \ref{Polycor2}.

\begin{lemma}\label{Lq}We have
$$\sum_{k=1}^{\infty}\frac{k I_k - (q^k-1)}{q^k-1} =0.$$
\end{lemma}
\begin{proof}
For $|z|<1$ we have \cite[p. 13]{Rosen}
\begin{equation*}
  \frac{1}{1-z} = \prod_{k\ge 1} \left(1-\frac{z^k}{q^k}\right)^{-I_k}.
\end{equation*}
Taking logarithms and differentiating yields
\begin{equation}\label{polyeuler}
\frac{1}{1-z} = \frac{1}{z} \sum_{k\ge 1}  \frac{k I_k}{q^k/z^k-1}. 
\end{equation}
Now write the left-hand side as $z^{-1}\sum_{k\ge 1} z^k$ and subtract to get
$$ 0=\sum_{k\ge 1}\frac{(kI_k-q^k+z^k)z^k}{q^k-z^k}.$$
If $|z|<\sqrt{q}$, the numerators in the last sum are $\ll q^{k/2}z^k$ by \eqref{In}, while the denominators are $\gg q^k$. 
Thus the last series converges uniformly on the disk $|z|<(1+\sqrt{q})/2$ and is therefore continuous at $z=1$, which is all we need.
\end{proof}

\section{An explicit upper bound for $P(x)$}

We first need an explicit upper bound for sums of powers of $\tau(n)$, the number of divisors of $n$.
Let $$\mu=-\log(\log 2 )/\log 2 = 0.528766...$$ and let $$\nu= 2^\mu -1 =1/\log 2 -1 =0.442695...$$ 
This choice of $\mu$ maximizes $\mu-\nu = \mu -2^\mu+1$.

\begin{lemma}\label{taulem}For $x\ge 2$, 
$$ \sum_{n\le x} (\tau(n))^\mu \le 1.315\, x (\log x)^{\nu} \left(1+\frac{0.2}{(\log x)^2}\right)^\nu.$$
\end{lemma}

\begin{proof}
Lemma 2.5 of Norton \cite{Nor} implies
\begin{multline*}
\sum_{n\le x} (\tau(n))^\mu  \le x \prod_{p\le x} \left( 1+ \sum_{k\ge 1} \frac{(k+1)^\mu - k^\mu}{p^k}\right) \\
= x  \prod_{p\le x} \left(1-\frac{1}{p}\right)^{-\nu}  
\prod_{p\le x} \left\{  \left(1-\frac{1}{p}\right)^{\nu}    \left( 1+ \sum_{k\ge 1} \frac{(k+1)^\mu - k^\mu}{p^k}\right) \right\}.
\end{multline*}
The second product clearly converges. With the help of a computer we find that it is less than $1.0181$ for all $x\ge 1$.
To estimate the first product, we use the following result by Dusart \cite[Theorem 6.12]{Dusart}: For $x\ge 2973$,
$$ \prod_{p\le x} \left(1-\frac{1}{p}\right)^{-1}  < e^\gamma \log x \left(1+\frac{0.2}{\log^2 x}\right).$$
Since $1.0181 e^{\gamma  \nu} < 1.315$, the result follows for $x\ge 2973$. For $x<2973$, we verify the lemma with a computer.
\end{proof}

\begin{lemma}\label{PUB}
For $x\ge 2$,
$\displaystyle P(x) \le  \frac{1.185 \, x}{ (\log x)^{ \mu -\nu}}.$
\end{lemma}

\begin{proof}
If $n$ is practical, then $2^{\tau(n)}\ge n$, since every natural number $m\le n$ can be expressed as a subsum of $\sum_{d|n} d $, and the number of  subsums is $2^{\tau(n)}$. Thus
$$ P(x) \le 1 + \sum_{2\le n \le x}\left( \tau(n) \frac{\log 2}{\log n} \right)^\mu.$$
Partial summation and Lemma \ref{taulem} yield
$$ P(x) \le 1.315 (\log 2)^\mu \frac{x}{(\log x)^{\mu-\nu}} f(x),$$
where
$$f(x) = \left(1+\frac{0.2}{(\log x)^2}\right)^\nu + \frac{\mu (\log x)^{\mu-\nu}}{x} 
\int_2^x \frac{ \left(1+\frac{0.2}{(\log t)^2}\right)^\nu}{(\log t)^{\mu-\nu+1}} dt. $$
Since $f(x)$ is decreasing for $x\ge 6$ and $1.315 (\log 2)^\mu f(1320)<1.185$, the result follows for $x\ge 1320$.
For $x<1320$, the trivial bound $P(x)\le x$ is sufficient.
\end{proof}
 
 \begin{lemma}\label{DNlem}
 If $P(x) \le a x (\log x)^{-b}$ for all $x \ge N$ and some constants $a, b>0$, then 
 $$\sum_{n>N} \frac{\chi(n)}{n\log n} ( \log_3 n - \log_3N) \le \frac{a(1+1/\log N)}{b^2 (\log N)^b \log \log N}.$$
 \end{lemma}

\begin{proof}
This is a standard exercise using partial summation and integration by parts. 
\end{proof}

\section*{Acknowledgments}
The author thanks the anonymous referee for several helpful suggestions.

The numerical calculations were performed with GP/PARI and Mathematica. These programs are available from the author upon request.

\end{document}